\documentclass[12pt,reqno]{amsart}
\let\oldlabel=\label
\def\prellabel{\marginparsep=1em\marginparwidth=44pt
    \def\label##1{\oldlabel{##1}\ifmmode\else\ifinner\else
         \marginpar{{\footnotesize\ \\ \tt
                    ##1}}\fi\fi}}
\def\NZQ{\mathbb}               
\def\NN{{\NZQ N }}

\def\ZZ{{\NZQ Z}}

\def\opn#1#2{\def#1{\operatorname{#2}}} 
\opn\chara{char}
\opn\rank{rank}
\opn\hilb{Hilb}
\opn\gr{gr}
\opn\Rees{{\mathcal R}}
\let\dirsum=\oplus

\usepackage[latin1]{inputenc}
\usepackage{mathptmx}
\newtheorem{theorem}{Theorem}[section]
\newtheorem{lemma}[theorem]{Lemma}
\newtheorem{corollary}[theorem]{Corollary}
\newtheorem{proposition}[theorem]{Proposition}

\theoremstyle{definition}
\newtheorem{remark}[theorem]{Remark}

\newtheorem{example}[theorem]{Example}
\newtheorem{remark/example}[theorem]{Remark/Example}

\newtheorem*{theorem*}{Theorem}

\let\epsilon=\varepsilon
\let\phi=\varphi
\let\kappa=\varkappa
\def\Tau{{\textup T}}

\opn\ini{in}
\opn\KRS{KRS}
\opn\krs{krs}
\opn\Krs{Krs}
\opn\DEL{DEL}
\opn\diag{diag}
\opn\Ker{Ker}
\opn\Image{Im}
\opn\DD{{\mathcal D}}
\opn\SS{{\mathcal S}}
\opn\MM{{\mathcal M}}
\opn\GL{GL}
\def\cF{{\mathcal F}}
\def\cA{{\mathcal A}}

\def\sep{\,|\,}
\def\addots{\mathinner{\mkern1mu\raise1pt\hbox{.}\mkern2mu\raise4pt\hbox{.}
        \mkern2mu\raise7pt\vbox{\kern7pt\hbox{.}}\mkern1mu}}

\def\Box#1#2#3{\multiput(#1,#2)(1,0){2}{\line(0,1)1}
                           \multiput(#1,#2)(0,1){2}{\line(1,0)1}
                           \put(#1,#2){\makebox(1,1){$#3$}}}

\def\LBox#1#2#3#4{\multiput(#1,#2)(#4,0){2}{\line(0,1)1}
                           \multiput(#1,#2)(0,1){2}{\line(1,0){#4}}
                           \put(#1,#2){\makebox(#4,1){$#3$}}}

\def\Ci#1#2{\put(#1.5,#2.5){\circle{0.7}}}

\numberwithin{equation}{section}

\author{Andrew Berget}
\address{Department of Mathematics, University of Washington, Seattle, USA}
\email{aberget@math.washington.edu}

\author{Winfried Bruns}
\address{Universit\"at Osnabr\"uck, Institut f\"ur Mathematik, 49069 Osnabr\"uck, Germany}
\email{wbruns@uos.de}

\author{Aldo Conca}
\address{ Dipartimento di Matematica,
Universit\`a degli Studi di Genova, Italy}
\email{conca@dima.unige.it}

\title{Ideals generated by superstandard tableaux}


\date{}

\begin{document}

\begin{abstract}
We investigate products $J$ of ideals of  ``row initial'' minors
in the polynomial ring $K[X]$ defined by a generic $m\times
n$-matrix. Such ideals are shown to be generated by a certain set of standard 
bitableaux that  we call superstandard.
These bitableaux form a Gr\"obner basis of $J$, and $J$ has a linear minimal
free resolution. These results are used to derive a new
generating set for the Grothendieck group of finitely generated
$T_m \times \GL_n(K)$-equivariant modules over $K[X]$. We
employ the Knuth--Robinson--Schensted correspondence and a
toric deformation of the multi-Rees algebra that parameterizes
the ideals $J$.
\end{abstract}

\maketitle
\section{Introduction}\label{Intro}
Let $K$ be a field and $X$ an $m\times n$ matrix of indeterminates
$x_{ij}$ over $K$. We write $R=K[X]$ for the polynomial ring in the
$x_{ij}$. The group $\GL_m(K) \times \GL_n(K)$ acts on $R$ with an
action induced by the rule $(g,g')\cdot X = g X g'^{-1}$. The
representation theory of $R$ as a module for this group is intimately
connected to the linear basis of $R$ given by \emph{bitableaux}
\cite[Ch.~11]{BV}. The bitableaux are products of minors which are
indexed by pairs of tableaux of the same shape with strictly
increasing rows and weakly increasing columns. We say that a bitableau
is \emph{superstandard} if its left factor tableaux has column $i$
filled with the number $i$. The left tableau determines the row
indices of the minors whose product the bitableau represents.

For each $i$, $1 \leq i \leq m$, let $J_i \subset R$ denote the ideal
generated by the size $i$ minors of the first $i$ rows of $X$. In the
current work we study an arbitrary product of such ideals. For a
decreasing sequence of positive integers $\min(m,n)\geq s_1 \geq \dots \geq s_\nu $
we set $J_S = J_{s_1} \dots J_{s_\nu}$. It is a consequence of
Theorem~\ref{superbasis} that the ideals $J_S$ are exactly those that
are generated by superstandard bitableau of shape $S$.

Our main results are Theorems~\ref{mainkrs} and \ref{mainsagbi}, which
we summarize here as follows.
\begin{theorem*}\leavevmode
\begin{enumerate}
\item The collection of superstandard bitableaux of shape
    $S$ in $R$ forms a Gr\"obner basis for the ideal $J_S$
    with respect to a diagonal monomial order.

\item  The ideal $J_S$ has a linear minimal free
    resolution.
\end{enumerate}
\end{theorem*}

The theorem is supplemented by results on primary
decompositions and integral closedness. Statement (1) will be
demonstrated in two ways. The first is via the
Knuth--Robinson--Schensted correspondence, and this approach,
together with a brief introduction to standard bitableaux, the
straightening law, and the KRS correspondence, will occupy
Sections~\ref{SectStr} and \ref{SectKRS}. The second proof of
(1) and the proof of (2) are via Sagbi (or toric) deformations.
It will take place in Section~\ref{sagbi}. The crucial point
for (2) is that the multi-Rees algebra of the ideals
$J_1,\dots,J_m$ is a Koszul algebra, and, in its turn, this
will be derived from the Koszul property of the initial algebra
of the multi-Rees algebra.

The theorem should be viewed as occurring in the greater
context of ideals generated by a family of bitableaux possessing
natural Gr\"obner bases \cite{BC1,BC2,Co4,Stu,StS,Sul}. Nevertheless,
the fact that the standard bitableaux in a product ideal like
$J_S$ form a Gr\"obner basis, is a rare phenomenon associated
with ideals generated by ``maximal'' minors. Statement (1) of
the theorem is a direct generalization of Conca's
result \cite{Co4} for rectangular shapes $S$.

In Section~\ref{SectK} we use statement (1) of the theorem to derive a
new generating set for the Grothendieck group of finitely generated
$T_m \times \GL_n(K)$-equivariant $R$-modules, where $T^m \subset
\GL_m(K)$ is the torus of diagonal matrices.  Having a basis for this
group coming from structure sheaves of schemes was the original
motivation for studying the class of ideals $J_S$.

\section{The straightening law}\label{SectStr}
Let $K$ be a field and $X=(x_{ij})$ an $m\times n$ matrix of
indeterminates $x_{ij}$ over $K$. We will study determinantal
ideals in the polynomial ring $R=K[X]=K[x_{ij}:i=1,\dots,m,\
j=1,\dots,n]$ generated by all the indeterminates $x_{ij}$.

Almost all of the approaches one can choose for the
investigation of determinantal ideals use standard bitableaux
and the straightening law. The principle governing this
approach is to consider all the minors of $X$ (and not just the
$1$-minors $x_{ij}$) as generators of the $K$-algebra $R$ so
that products of minors appear as ``monomials''. The price to
be paid, of course, is that one has to choose a proper subset
of all these ``monomials'' as a linearly independent $K$-basis:
the standard bitableaux to be defined below are a natural
choice for such a basis, and the straightening law tells us how
to express an arbitrary product of minors as a $K$-linear
combination of the basis elements. (In \cite{BV} standard
bitableaux were called \emph{standard monomials}; however, we
will have to consider the ordinary monomials in $K[X]$ so often
that we reserve the term ``monomial'' for products of the
$x_{ij}$.)

In the following
$$
[a_1,\dots,a_t\sep b_1,\dots,b_t]
$$
stands for the determinant of the submatrix $(x_{a_ib_j}:
i=1,\dots,t,\ j=1,\dots,t)$.

The letter $\Delta$ always denotes a product
$\delta_1\cdots\delta_w$ of minors, and we assume that the
sizes $|\delta_i|$ (i.~e.\ the number of rows of the submatrix
$X'$ of $X$ such that $\delta_i=\det(X')$) are
\emph{descending}, $|\delta_1|\ge \dots\ge |\delta_w|$. By
convention, the empty minor $[\sep]$ denotes $1$. The
\emph{shape} $|\Delta|$ of $\Delta$ is the sequence
$(|\delta_1|,\dots,|\delta_w|)$. If necessary we may add
factors $[\sep]$ at the right hand side of the products, and
extend the shape accordingly.

A product of minors is also called a \emph{bitableau}. The
choice of this term ``bitableau'' is motivated by the graphical
description of a product $\Delta$ as a pair of Young tableaux
as in Figure \ref{Young}:
\begin{figure}[hbt]
\begin{gather*} \unitlength=0.8cm
\begin{picture}(6,4)(0,0)
\Box03{a_{1t_1}} \LBox13{\cdots}4 \Box53{a_{11}} \Box12{a_{2t_2}}
\LBox22{\cdots}3 \Box52{a_{21}} \put(2,1){\makebox(4,1){$\cdots$}}
\Box20{a_{wt_w}} \LBox30{\cdots}2 \Box50{a_{w1}}
\end{picture} \hspace{0.5cm}
\begin{picture}(6,4)(0,0)
\Box03{b_{11}} \LBox13{\cdots}4 \Box53{b_{1t_1}} \Box02{b_{21}}
\LBox12{\cdots}3 \Box42{b_{2t_2}} \put(0,1){\makebox(4,1){$\cdots$}}
\Box00{b_{w1}} \LBox10{\cdots}2 \Box30{b_{wt_w}}
\end{picture}
\end{gather*}
\caption{A bitableau}\label{Young}
\end{figure}
Every product of minors is represented by a bitableau and,
conversely, every bitableau stands for a product of minors if
the length of the rows is decreasing from top to bottom, the
entries in each row are strictly increasing from the middle to
the outmost box, the entries of the left tableau are in
$\{1,\dots,m\}$ and those of the right tableau are in
$\{1,\dots,n\}$. These conditions are always assumed to hold.

For formal correctness one should consider the bitableaux as
purely combinatorial objects and distinguish them from the
ring-theoretic objects represented by them, but since there is
no real danger of confusion, we simply identify them.

Whether $\Delta$ is a standard bitableau is controlled by a
partial order of the minors, namely
\begin{multline*}
[a_1,\dots,a_t\sep b_1,\dots,b_t] \leq [c_1,\dots,c_u\sep d_1,\dots,d_u]\\
\iff\quad t\ge u\quad\text{and}\quad a_i\le c_i,\ b_i\le d_i,\
i=1,\dots,u.
\end{multline*}
A product $\Delta=\delta_1\cdots\delta_w$ is called a
\emph{standard bitableau} if
$$
\delta_1\leq\dots\leq\delta_w,
$$
in other words, if in each column of the bitableau the indices
are non-decreasing from top to bottom. The letter $\Sigma$ is
reserved for standard bitableaux.

The fundamental straightening law of Doubilet--Rota--Stein says
that every element of $R$ has a unique presentation as a
$K$-linear combination of standard bita\-bleaux (for example,
see Bruns and Vetter \cite{BV})

\begin{theorem}\label{straight}
\begin{itemize}
\item[(a)] The standard bitableaux are a $K$-vector space
    basis of $K[X]$.
\item[(b)] If the product $\delta_1\delta_2$ of minors is
    not a standard bitableau, then it has a representation
$$
\delta_1\delta_2=\sum x_i\epsilon_i\eta_i,\qquad x_i\in K,\ x_i\neq0,
$$
where $\epsilon_i\eta_i$ is a standard bitableau for all
$i$ and $\epsilon_i<\delta_1,\delta_2<\eta_i$ (here we must
allow that $\eta_i=1$).
\item[(c)] The standard representation of an arbitrary
    bitableau $\Delta$, i.e., its representation as a
    linear combination of standard bitableaux $\Sigma$, can
    be found by successive application of the straightening
    relations in \emph{(b)}.
\end{itemize}
\end{theorem}

Let $e_1,\dots,e_m$ and $f_1,\dots,f_n$ denote the canonical
$\ZZ$-bases of $\ZZ^m$ and $\ZZ^n$ respectively.  Clearly
$K[X]$ is a $\ZZ^m\dirsum\ZZ^n$-graded algebra if we give
$x_{ij}$ the ``vector bidegree'' $e_i\dirsum f_j$. All minors
are homogeneous with respect to this grading. In a bitableau of
bidegree $(c_1,\dots,c_m,d_1,\dots,d_n)\in\ZZ^m\dirsum\ZZ^n$,
row $i$ appears with multiplicity $c_i$, and column $j$ appears
with multiplicity $d_j$, $i=1,\dots,m$, $j=1,\dots,n$. The
straightening relations must therefore preserve these
multiplicities, whose collection is often called the
\emph{content} of the bitableau.

We say that an ideal $I\subset R$ has a \emph{standard basis}
if $I$ is the $K$-vector space spanned by the standard
bitableaux $\Sigma\in I$.

Let $S=s_1,\dots,s_v$ be  weakly decreasing sequence of
positive integers $s_i\le \min(m,n)$. In this article we
investigate the ideal
$$
J_S=J_{s_1}\cdots J_{s_v}
$$
where $J_t$ is the ideal generated by the $t$-minors of the
first $t$ rows of $X$. In other words, $J_t$ is the the ideal
of maximal minors of the matrix $X_t$ formed by the first $t$
rows of $X$ in $K[X_t]$ and extended to $K[X]$. We will see
that the ideals $J_S$ behave very much like the powers of
ideals of maximal minors that they generalize in a natural way.

The bitableaux $\Delta=\delta_1\cdots\delta_v$ with
$\delta_i\in J_{s_i}$, $|\delta_i|=s_i$, are automatically
standard on the left side (the tableau of row indices). We call
them \emph{row superstandard} and just \emph{superstandard} if
they are also standard on the right side. Note that in a (row)
superstandard bitableau all indices $a_{ij}$ are as small as
possible, namely $a_{ij}=j$. In \cite{BV} superstandard
tableaux are called \emph{row initial}, but we want to reserve
the term ``initial'' for use in connection with monomial
orders.

Let $\Delta=\delta_1\cdots\delta_u$ and
$\Delta'=\delta_1'\cdots\delta_w'$ be bitableaux. We say that
$\Delta'$ is a \emph{subtableau} of $\Delta$ if $w\le u$,
$|\delta_i'|\le|\delta_i|$ for $i=1,\dots,w$ and, with
$s=|\delta_i|$, $t=|\delta'_i|$, and $\delta=[a_{i1}\dots
a_{is}\sep b_{i1}\dots b_{is}]$ one has
$$
\delta_i'=[a_{i1}\dots a_{it}\sep b_{i1}\dots b_{it}]
$$
for $i=1,\dots,w$. Subtableaux of (super)standard bitableaux
are evidently (super)standard.

\begin{theorem}\label{superbasis}
The ideal $J_S$ has a standard basis that is given by all
standard bitableaux containing a superstandard tableau of shape
$S$.
\end{theorem}

\begin{proof}
As a vector space over $K$, $J_S$ is certainly generated by all
products
$$
\delta_1\cdots\delta_w,\qquad w\ge v,
$$
such that $\delta_i=[1 \dots s_i \dots \sep \dots]$ for
$i=1,\dots,v$. (We do not assume that the $\delta_i$ are
ordered by size.) It is enough to show that this property is
preserved by all products of minors that arise if we replace an
incomparable subproduct $\delta_i\delta_j$ by the right hand
side of the straightening relation.

Let $\delta_i=[1 \dots s_i \dots \sep \dots]$ and $\delta_j=[1
\dots s_i \dots \sep \dots]$ where we have set $s_j=0$ if
$j>v$. It is immediately clear that the first factor $\epsilon$
of each summand on the right hand side of the straightening
relation must be of type $[1 \dots s_i \dots\sep\dots]$ since
$\epsilon\le \delta_i$, and since no index is lost on the right
hand side, the second factor satisfies $\eta=[1\dots
s_j\dots\sep\dots]$.

After finitely many steps we arrive at a $K$-linear combination
of standard bitableaux, each of which contains a superstandard
tableau of shape $S$.
\end{proof}

The description of the standard basis yields the primary
composition of the ideals $J_S$ as an easy consequence:

\begin{corollary}\label{primary}
Write $\{s_1,\dots,s_v\}=\{t_1,\dots,t_u\}$ with
$t_1>\dots>t_u$ and set $e_i=\max\{j: s_j=t_i\}$. Then
$$
J_S=\bigcap_{i=1}^u J_{t_i}^{e_i}
$$
is an irredundant primary decomposition, and $J_S$ is an
integrally closed ideal.
\end{corollary}

\begin{proof}
The ideals on both sides have a standard basis as follows from
the theorem. Therefore it is enough to compare these. But a
standard bitableau contains a superstandard bitableau of shape
$S$ if and only if it contains a rectangular superstandard
bitableau with $e_i$ rows of length $t_i$ for every $i$, and
the latter form the standard basis of $J_{t_i}^{e_i}$ by the
theorem.

Comparing standard bases once more, we see that none of the
$J_{t_i}^{e_i}$ is contained in the intersection of the others.

Finally, it remains to observe that the ideals $J_{t_i}^{e_i}$
are primary. But $J_{t_i}^{e_i}$ arises from
$I_{t_i}(X_{t_i})^{e_i}$ by tensoring over $K$ with the
polynomial ring in the variables $x_{kl}$ outside $X_{t_i}$,
and such extensions preserve the property of being primary.
That the powers of $I_{t_i}(X_{t_i})$ are primary is
well-known; see \cite[9.18]{BV}.

For the last statement it is enough to note that the powers
$J_{t_i}^{e_i}$ are not only primary, but also integrally
closed. This follows from the normality of the Rees algebra
$\Rees(J_{t_i})$ \cite[9.17]{BV}.
\end{proof}

The statement on integral closedness is equivalent to the
normality of a multi-Rees algebra. We postpone this aspect
until Theorem \ref{mainsagbi}.

\section{The Knuth--Robinson--Schensted correspondence}\label{SectKRS}

Let $\Sigma$ be a standard bitableau. The
Knuth--Robinson--Schensted correspondence KRS (see Fulton
\cite{F} or Stanley \cite{Sta}) sets up a bijective
correspondence between standard bitableaux and monomials in the
ring $K[X]$. The treatment of KRS below follows \cite{BC1} and
\cite{BC2}. However, for better compatibility with the
definition of the ideals $J_S$ we have exchanged the roles of
the left and right tableau.

If one starts from bitableaux, the correspondence is
constructed from the algorithm \emph{KRS-step} \cite[4.2]{BC2}
(based on \emph{deletion} \cite[4.1]{BC2}). Let $\Sigma=(a_{ij}
| b_{ij} )$ be a non-empty standard bitableau. The output of
KRS-step is a triple $(\Sigma',\ell,r)$ consisting of a
standard bitableau $\Sigma'$ and a pair of integers $(\ell,r)$
constructed as follows.
\begin{itemize}
\item[(a)] One chooses the largest entry $r$ in the
    \emph{right} tableau of $\Sigma$; suppose that $\{
    (i_1,j_1),\allowbreak \dots,\allowbreak (i_u,j_u)\}$,
    $i_1<\dots<i_u$, is the set of indices $(i,j)$ such
    that $r=b_{ij}$. (Note that $j_1\ge\dots\ge j_u$.)
\item[(b)] Then the boxes at the \emph{pivot position}
    $(p,q)=(i_u,j_u)$ in the right and the left tableau are
    removed.
\item[(c)] The entry $r=b_{pq}$ of the removed box in the
    right tableau is the third component of the output, and
    $a_{pq}$ is stored in $s$, an auxiliary memory cell.
 \item[(d)] The first and the second component of the
     output are determined by a ``push out'' procedure on
     the \emph{left} tableau as follows:
\begin{itemize}
\item[(i)]  if $p=1$, then $\ell=s$ is the second
    component of the output, and the first is the
    standard bitableau $\Sigma'$ that has now been
    created;
\item[(ii)] otherwise $s$ is moved one row up and
    pushes out the left most entry $a_{p-1k}$ such that
    $a_{p-1k}\le s$ whereas $a_{p-1k}$ is stored in
    $s$.
\item[(iii)] one replaces $p$ by $p-1$ and goes to step
    (i).
\end{itemize}
\end{itemize}
It is now possible to define KRS recursively: One sets
$\KRS([\sep])=1$, and $\KRS(\Sigma)=\KRS(\Sigma')x_{\ell r}$
for $\Sigma\neq[\sep]$.

There is an inverse to deletion, called insertion that can be
easily constructed by inverting all steps in deletion. Together
they prove the main theorem on KRS:

\begin{theorem}\label{KRSmain}
The map $\KRS$ is a bijection between the set of standard
bitableaux on $\{1,\dots,m\}\times\{1,\dots,n\}$ and the
monomials of $K[X]$.
\end{theorem}

For insertion one must order the factors of a monomial in a way
that respects the monotonicity properties of KRS-step: let
$x_{r_1\ell_1}\cdots x_{r_k\ell_k}=\KRS(\Sigma)$ with the
factors ordered as in the definition of KRS; then
\begin{equation*}
r_i\le r_{i+1}\qquad\text{and}\qquad r_i=r_{i+1} \implies\ell_i\ge \ell_{i+1}.
\eqno{(*)}
\end{equation*}
See \cite[p.~37]{BC2} (with $r$ and $\ell$ exchanged). Property
$(*)$ allows us to take care of a superstandard subtableau, but
some additional bookkeeping is necessary. To this end we extend
the output of KRS-step by a further component $\rho$, the
\emph{row mark} that we will now define. (Here ``row'' refers
to the tableau, not to a minor.)

Let $S=s_1,\dots,s_v$ a nonincreasing sequence as above, and
suppose that $\Sigma$ contains a superstandard bitableau of
shape $S$. Then we can distinguish boxes in the left tableau
that belong to the superstandard bitableau from those that do
not belong to it, namely the box at position $(i,j)$ belongs to
the superstandard subtableau  if and only if $a_{ij}=j$ and
$j\le s_i$. We supplement step (d) above by
\begin{itemize}
\item[(iv)] if $a_{ij}=j$ and $j\le s_i$, but $(i,j)$ is
    the pivot position or $a_{ij}'>a_{ij}$, then $\rho=i$
    is the fourth component of the output of KRS-step.
    Otherwise we set $\rho=0$.
\end{itemize}

Let us first make sure that rule (iv) makes sense by showing
that there can be at most one row $i$ with $a_{ij}=j$ and
$a_{ij}'> a_{ij}$. This is clear if $(i,j)$ is the pivot
position since all remaining positions remain unchanged. In the
other case, if $a_{ij}=j$ and $a_{ij}'>a_{ij}$, then
$i=\max\{k:a_{kj}\}=j$. In fact, if the box at position $(i,j)$
is hit by the push out sequence in KRS-step(d) and $a_{ij}=j$,
then the entry $j$ is pushed out into the next upper row and
replaces $a_{i-1j}=j$ by $j$.

The triples $(\ell,r,\rho)$ form the columns of a three row
array $\krs(\Sigma)$ that we build by listing the triples
$(\ell,r,\rho)$ from right to left as follows:
$$
\krs(\Sigma)=\krs(\Sigma')\begin{pmatrix}\ell\\ r\\ \rho\end{pmatrix}.
$$

We give an example in Figure \ref{Del} with $S=3,2$. The
circles in the right tableau mark the pivot position, those in
the left mark the chains of ``pushouts'':

\begin{figure}[hbt]
$$
\begin{gathered}
\begin{picture}(3,3)(0,0)
\Box221 \Box122 \Box023 \Box211 \Box112 \Box014 \Box202
\Ci20 \Ci11 \Ci12
\end{picture}
\hspace{0.5cm}
\begin{picture}(3,3)(0,0)
\Box021 \Box122 \Box223 \Box012 \Box113 \Box214 \Box004
\Ci00
\end{picture}
\\[0.5cm] 
\begin{picture}(3,2)(0,0)
\Box211 \Box112 \Box013 \Box201 \Box102 \Box004
\Ci 01 \Ci00
\end{picture}
\hspace{0.5cm}
\begin{picture}(3,2)(0,0)
\Box011 \Box112 \Box213 \Box002 \Box103 \Box204
\Ci20
\end{picture}
\\[0.5cm] 
\begin{picture}(3,2)(0,0)
\Box211 \Box112 \Box014 \Box201 \Box102
\Ci 11
\Ci 10
\end{picture}
\hspace{0.5cm}
\begin{picture}(3,2)(0,0)
\Box011 \Box112 \Box213 \Box002 \Box103
\Ci10
\end{picture}
\vspace*{0.5cm} 
\end{gathered}
\qquad\qquad
\begin{gathered}
\begin{picture}(3,2)(0,0)
\Box014 \Box112  \Box 211 \Box201 \Ci01
\end{picture}
\hspace{0.5cm}
\begin{picture}(3,2)(0,0)
\Box011 \Box112 \Box213 \Box002 \Ci21
\end{picture}
\\[0.5cm] 
\begin{picture}(2,2)(0,0)  
\Box111 \Box012 \Box101
\Ci10 \Ci11
\end{picture}
\hspace{0.5cm}
\begin{picture}(2,2)(0,0)
\Box011 \Box112 \Box002
\Ci00
\end{picture}
\\[0.5cm] 
\begin{picture}(2,1)(0,0) 
\Box002 \Box101
\Ci00
\end{picture}
\hspace{0.5cm}
\begin{picture}(2,1)(0,0)
\Box001 \Box102
\Ci10
\end{picture}
\\[0.5cm] 
\begin{picture}(2,2)(0,0)
\Box111
\Ci11
\end{picture}
\hspace{0.5cm}
\begin{picture}(2,2)(0,0)
\Box011
\Ci01
\end{picture}
\end{gathered}
$$
\vspace*{-1cm}
\caption{The KRS algorithm}\label{Del}
\end{figure}
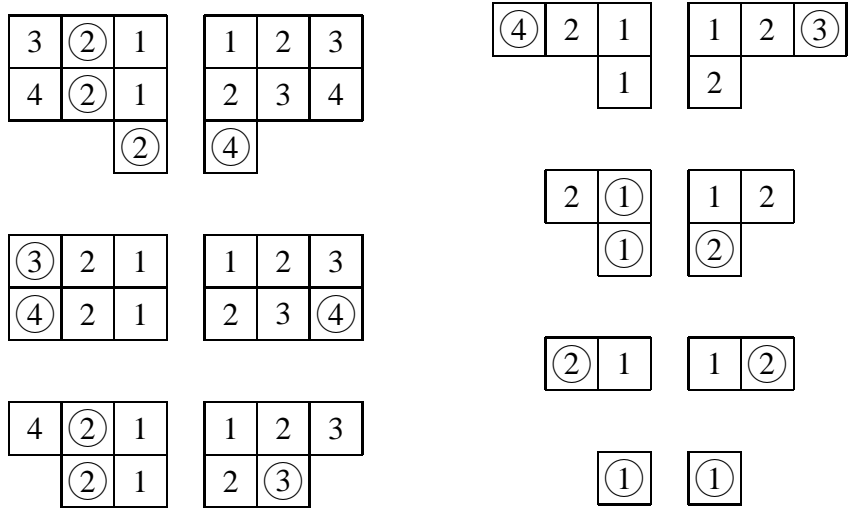

The three row array produced by  the example of Figure
\ref{Del} is
$$\krs(\Sigma)=\begin{pmatrix}
1&2&1&4&2&3&2\\
1&2&2&3&3&4&4\\
1&1&2&0&2&1&0
\end{pmatrix},
$$
and
$$
\KRS(\Sigma)=x_{11}x_{22}x_{12}x_{23}x_{44}x_{34}x_{24}.
$$

Let us extract the subarrays with row marks $1$ and $2$:
$$
\begin{pmatrix}
1&2&3\\
1&2&4\\
1&1&1
\end{pmatrix}
\qquad\text{and}\qquad
\begin{pmatrix}
1&2\\
2&3\\
2&2
\end{pmatrix}
$$
The product of the corresponding monomials
$$
x_{11}x_{22}x_{34}\qquad\text{and}\qquad x_{12}x_{23}
$$
is the KRS image of a superstandard bitableau of shape $(3,2)$
(though it is not the KRS image of the superstandard subtableau
contained in $\Sigma)$. What we have observed in this special
case is always true, as we will be stated in Lemma
\ref{crucial} below.

Let
$$
\diag [a_1\dots a_t\sep b_1\dots b_t]=\prod_{i=1}^t x_{a_i b_i}
$$
be the product of the indeterminates in the main diagonal of
$[a_1\dots a_t\sep b_1\dots b_t]$. If
$\Delta=\delta_1\cdots\delta_w$ is an arbitrary bitableau, then
we set
$$
\diag(\Delta)=\prod_{i=1}^w \diag(\delta_i).
$$

It is easy to see that the map $\diag$ is not injective on
standard bitableaux (let alone all bitableaux), in contrast to
KRS. (Otherwise KRS would be completely superfluous in the
study of determinantal ideals.) However, if $\Sigma$ is a
superstandard bitableau, then
\begin{equation}
\diag(\Sigma)=\KRS(\Sigma)\label{KRSdiag}
\end{equation}
since the whole push out sequence in KRS-step(d) always
replaces the entry of a box by itself. 

\begin{lemma}\label{crucial}
Let $\Sigma$ be a standard bitableau containing a superstandard
bitableau of shape $S$. Then there exists a superstandard
bitableau $\Tau$ of shape $S$ such that $\diag(\Tau)$ divides
$\KRS(\Sigma)$.
\end{lemma}

\begin{proof}
Suppose $\Tau$ is a (not necessarily standard) bitableau whose
row tableau is superstandard of shape $S$. Then
$\diag(\Tau)=\diag(\Tau')$ where $\Tau'$ is standard of shape
$S$. This is easy to see and left to the reader. Therefore it
is enough to prove the lemma without the requirement that
$\Tau$ is standard. (Equation \eqref{KRSdiag} would allow us to
replace $\diag(\Tau)$ by $\KRS(\Tau)$, but this is irrelevant.)

Let $\Sigma=(a_{ij}\sep b_{ij})$ and choose an index $k$ such
that row $k$ of $\Sigma$ occurs in the superstandard
subtableau. Let $s=\max\{j: a_{kj}=j\}$. As in the example we
extract the subarray $A$ from $\krs(\Sigma)$ with row mark $k$.
We claim that the corresponding monomial is the diagonal of an
$s$-minor $[1\dots s\sep c_1 \dots c_s]$. This claim amounts to
the following conditions for the subarray~$A$:
\begin{enumerate}
\item The entries of the first row are $1,\dots,s$ in
    ascending order;
\item the entries $c_1,\dots,c_s$ of the second row are
    strictly increasing.
\end{enumerate}

First of all we note that the row mark is $k$ if a box at
position $(k,z)$ with $a_{kz}=z$ changes its content in
KRS-step: either $(k,z)$ is the pivot position or in
$\Sigma'=(a_{ij}'\sep b_{ij}')$ one has $a_{kz}'>z$. This
change happens exactly once for $j=1,\dots,s$. Therefore the
entries of the first row of $A$ are indeed $1,\dots,s$.

But $1,\dots,s$ are also produced in the right order. If
$a_{kz}=z$, then $a_{kw}=w$ for $w=1,\dots,z-1$, and so these
boxes have not yet changed content. Moreover the component $r$
of the output of KS-step is exactly $z$, and $1,\dots,z-1$ will
be produced later. This proves (1).

The entries $c_1,\dots,c_s$ in the second row are automatically
weakly increasing by the first inequality in $(*)$, and an
equality of two entries would contradict (1) because of the
second inequality in $(*)$. In other words, (1) implies (2).
\end{proof}

It is now time to introduce a \emph{diagonal monomial} (or
\emph{term}) \emph{order} $\prec$ on the polynomial ring
$K[X]$. This is a term order on the polynomial ring under which
the initial monomial of each minor is the product of the
elements in the main diagonal:
$$
\ini_\prec [a_1\dots a_t\sep b_1\dots b_t]=\diag [a_1\dots a_t\sep b_1\dots b_t].
$$
Diagonal monomial orders are the standard choice in the study
of determinantal ideals from the Gr\"obner basis viewpoint. See
\cite{BC2} for a survey that also contains a brief introduction
to general Gr\"obner bases and initial ideals.

\begin{theorem}\label{mainkrs}
Let $S=s_1\dots s_u$ be a nonincreasing sequence. Then the
following hold:
\begin{enumerate}
\item the row superstandard bitableaux of shape $S$ form a
    Gr\"obner basis of $J_S$.
\item In particular, $\ini_\prec(J_S)=\KRS(J_S)$.
\item Furthermore, $\ini_\prec(J_S)=\prod
    \ini_\prec(J_{s_i})$, and
\item $\ini_\prec(J_S)=\bigcap_{i=1}^u
    \ini_\prec(J_{t_i}^{e_i})=\bigcap_{i=1}^u
    \ini_\prec(J_{t_i})^{e_i}$ where the sequences
    $\{t_1,\dots,t_u\}$ and $e_1,\dots,e_u$ are defined as
    in Corollary \ref{primary}.
\end{enumerate}
\end{theorem}
\begin{proof}
Claims (1) and (2) result immediately from Lemma \ref{crucial}
and \cite[Lemma 5.2]{BC2}.

Since $\prod \ini_\prec(J_{s_i})\subset \ini_\prec(J_S)$ for
obvious reasons, it is enough to observe the converse for (3).
But this follows again from Lemma \ref{crucial} since
$\ini_\prec(\Tau)$ is contained in $\prod \ini_\prec(J_{s_i})$.

In the terminology of \cite{BC1} or \cite{BC2}, claim (2),
applied to the sequence $t_i,\dots, t_i$ ($e_i$ repetitions)
says that the ideal $J_{t_i}^{e_i}$ are in-KRS, and for in-KRS
ideals the formation of initial ideals commutes with
intersection; see \cite[Lemma 5.2]{BC2}. So it remains to use
Corollary \ref{primary}.
\end{proof}

\section{Sagbi deformation}\label{sagbi}

Sagbi bases are the {\bf S}ubalgebra {\bf A}nalog of  {\bf
G}r\"obner bases for {\bf I}deals.  They have been introduced
by Robbiano and Sweedler \cite{RS}.  In \cite{CHV} Conca,
Herzog and Valla shown how to use Sagbi bases and Sagbi
deformation (also called toric deformation) in the study of
homological properties of subalgebras of polynomials rings and,
in particular, to Rees algebras.

In this section we will use Sagbi deformations of Rees algebras
to study the ideals $J_S$ defined in the previous sections. By
definition, these ideals are products of powers of the ideas
$J_1,\dots,J_m$ (we do not assume that $n\ge m$; if $m>n$ then  all results
in this section hold with $J_{n+1}=0,\dots,J_m=0$.)

Before we turn to our class of ideals we study the Sagbi
approach via Rees algebras in general. Let $A=K[x_1,\dots,x_r]$
the polynomial ring in $r$ indeterminates, endowed with a
monomial order $\prec$. For every $K$-vector subspace $V$ of
$A$ we may consider the vector space $\ini_\prec(V)$ generated
by the monomials $\ini_\prec(f)$ as $f\neq 0$ varies in $V$. If
$V$ is an ideal of $A$, then $\ini_\prec(V)$ will be an ideal
of $A$, and if $V$ is a $K$-subalgebra of $A$, then $\ini_\prec(V)$
will be a $K$-subalgebra of $A$ as well. If $V$ is an ideal, then a
subset $G$ of $V$ is a Gr\"obner basis if $\ini_\prec(V)$ is
generated (as an ideal) by $\{ \ini_\prec(f) : f \in G\}$.
Similarly, if $V$ is an algebra, then a subset $G$ of $V$ is a
\emph{Sagbi basis} if $\ini_\prec(V)$ is generated (as a
$K$-algebra) by $\{ \ini_\prec(f) : f \in G\}$. A variation of
the Buchberger criterion allows us to detect whether a given
set $G$ of polynomials is a Sagbi basis. One has to replace the
so called S-pairs with  the binomial relations defining the
toric ring $K[ \ini_\prec(f) : f \in G]$. We refer the reader
to \cite{CHV} for further details.

Let now $I_1,\dots,I_v$ homogeneous ideals of $A$. We want to
express the condition
\begin{equation}
\label{1}
\ini_\prec(I_1^{a_1}\cdots I_v^{a_v} )= \ini_\prec(I_1)^{a_1} \cdots
\ini_\prec(I_v)^{a_v}   \mbox{ for all } (a_1,\dots, a_v)\in \NN^v
\end{equation}
in terms of Sagbi deformations. Let
$$
\Rees(I_1,\dots,I_v)=\bigoplus_{a\in \NN^v}  I_1^{a_1}\cdots I_v^{a_v}
$$
be the (multi-)Rees ring $\Rees(I_1,\dots,I_v)$ associated to
the family $I_1,\dots,I_v$. In order to describe it as a as a
subalgebra of a polynomial ring, we take new variables
$y_1,\dots, y_v$. Then we can identify $\Rees(I_1,\dots,I_v)$
with the subalgebra
$$
A[I_1y_1,\dots, I_vy_v]\subset A[y]=A[y_1,\dots,y_v].
$$
By construction, $\Rees(I_1,\dots,I_v)$ has a $\ZZ\oplus
\ZZ^{v}$-graded structure induced by the assignment
$\deg(x_i)=e_0$ for all $i$ and $\deg(y_j)=e_j$ for all $j$
where $e_0,e_1,\dots,e_v$ denotes the canonical basis of
$\ZZ\oplus \ZZ^{v}$.

We extend $\prec$ to a monomial order on $K[x,y]$. It is indeed
irrelevant which extension is chosen because the polynomials we
will consider are ``monomial'' in the $y$'s and so we denote
the extension by $\prec$ as well.

Then
$$
\ini_\prec( \Rees(I_1,\dots,I_v))=
\bigoplus_{a\in \NN^v} \ini_\prec( I_1^{a_1}\cdots I_v^{a_v}),
$$
and hence \eqref{1} holds if and only if
\begin{equation}
\label{2}
\ini_\prec( \Rees(I_1,\dots,I_v))=
\Rees(\ini_\prec(I_1),\dots,\ini_\prec(I_v)).
\end{equation}
Condition \eqref{2} can be expressed in terms of Sagbi basis.

For every $i$ let $F_{i1},\dots, F_{ic_i}$ a Gr\"obner basis of
$I_i$ with respect to $\prec$.  As a $K$-algebra, the Rees ring
$\Rees(I_1,\dots,I_v)$ is generated by two sets of polynomials:
\begin{itemize}
\item[(1)] $X=\{x_1,\dots, x_r\}$ and
\item[(2)] $\cF=\{ F_{ij}y_i  :   i=1,\dots, v \mbox{ and }
    j=1,\dots, c_i\}$.
\end{itemize}

Condition \eqref{2}  is equivalent to the statement
\begin{equation}
\label{3}
X\cup \cF  \mbox{ is a Sagbi basis with respect to }  \prec.
\end{equation}
To test whether condition  \eqref{3} holds we can use the Sagbi
variant of the Buchberger criterion \cite{CHV}. Set
$$
M_{ij}=\ini_\prec(F_{ij}).
$$
and consider two $A$-algebra maps from the polynomial ring
$$
P=A[ p_{ij}  :  i=1,\dots, v,\ j=1,\dots, c_i]
$$
to $A[y]$ defined as follows:
$$
\Phi(p_{ij})=M_{ij}y_i\qquad\text{and}\qquad
\Psi(p_{ij})=F_{ij}y_i.
$$
By construction
$$
\Image \Phi= \Rees(\ini_\prec(I_1),\dots,\ini_\prec(I_v))
\qquad\text{and}\qquad
\Image \Psi= \Rees(I_1,\dots,I_v).
$$

The kernel of $\Phi$ is a toric ideal, i.e., a prime ideal
generated by binomials since
$\Rees(\ini_\prec(I_1),\dots,\ini_\prec(I_v))$ is a $K$-algebra
generated by monomials. These binomials replace the S-pairs in
the Buchberger criterion for Gr\"obner bases. Roughly speaking,
the following criterion says that every such binomial relation
of the initial monomials can be ``lifted'' to a relation of the
elements of $G$ themselves.

\begin{lemma}[Sagbi version of the Buchberger criterion] \label{SagBuch}
Let $G$ be a set of binomials generating $\Ker \Phi$.  Suppose that for every $g\in G$ such that $\Psi(g)\neq 0$
one has:
$$
\Psi(g)= \sum \lambda_{a,b} X^a\cF^b
$$
where $\lambda_{a,b}\in K^*$, and $X^a{\cF}^b$ is a monomial in
the set $X\cup \cF$ such that $\ini_\prec(X^a{\cF}^b) \preceq
\ini_\prec(\Psi(g))$ for all $a,b$.

Then $X\cup \cF$ is a Sagbi basis.
\end{lemma}

\begin{remark}
\label{SagBuch1} If $g$ has total degree $1$ in the $p_{ij}$'s,
then  the condition required in Lemma \ref{SagBuch} is
automatically satisfied because $F_{i1},\dots, F_{ic_i}$  is a
Gr\"obner basis of the ideal $I_i$.  So we have only to worry
about the $g\in G$ of degree $>1$ in the $p_{ij}$'s.
\end{remark}

Assume now that each ideal $I_i$ is generated in a single
degree, say $d_i$. Then  $\Rees(I_1,\dots,I_v)$ can be given
the structure of a standard $\ZZ\oplus \ZZ^{v}$-graded
$K$-algebra by assigning the degree $e_j-d_je_0$  to $y_j$,
$j=1,\dots,v$, and $e_0$ to the variables $x_i$. On $P$ we
define the grading by  $\deg(x_i)=e_0$ and $\deg(p_{ij})=e_i$.
Then the maps $\Phi$ and $\Psi$ are $\ZZ\oplus \ZZ^{v}$-graded.

The following theorem relates a ring theoretic property of the
Rees algebra to the free resolutions of the ideals involved:

\begin{theorem}[Blum]
\label{Kos} If  each $I_i$ is generated in a single degree and
$\Rees(I_1,\dots,I_v)$ is a Koszul algebra (for example, it is
defined by a Gr\"obner basis of quadrics) then $I_1^{a_1}\cdots
I_v^{a_v}$ has a linear resolution for all $a_1,\dots,a_v\in
\NN$.
\end{theorem}

This was proved by Blum \cite[Cor. 3.6]{B} for $v=1$, but the
proof generalizes immediately to the multigraded setting.

Now we return to the family of determinantal ideals  we are
interested in. Let $R=K[X]$ where $X=(x_{ij})$ is an $m\times n$-matrix of
indeterminates as introduced in Section \ref{SectStr}. For the
ideals of minors considered in this article, the equality
\eqref{1} is part of Theorem \ref{mainkrs}, but it will be
proved independently by the Sagbi approach.  Recall that, by definition,
for $t=1,\dots,m$ we denote by $J_t$ is the ideal generated by the $t$-minors of  the first $t$ rows of $X$.
For a nonincreasing
sequence $S=s_1,\dots,s_v$ the ideal $J_S=J_{s_1}\cdots J_{s_v}$ can be written as a
product of powers of the ideals $J_t$, and in this section it
is more convenient to use the latter representation.  To simplify notation we omit the row
indices in a superstandard tableau by setting
$$
[a_1\dots a_s]=[1\dots s\sep a_1\dots a_s].
$$
We know by Theorem \ref{mainkrs} that the minors   $[a_1\dots
a_s]$ are a Gr\"obner basis of $J_s$ with respect to a diagonal
monomial order. For the application of Lemma \ref{SagBuch} below we
set
\begin{itemize}
\item[(1)] $X=\{ x_{ij} : 1\leq i\leq m \mbox{ and } 1\leq j\leq n \}$,
\item[(2)] $\cF=\{ [a_1,\dots,a_s]y_s : 1\leq s\leq m \mbox{
    and } 1\leq a_1<\dots<a_s\leq n\}$.
\end{itemize}

Let
$$
{\cA}=\{ [a_1\dots a_s]: \ 1\le s\le m \mbox{ and } 1 \le a_1<\dots<a_s\leq n\}.
$$
The set $\cA$ inherits the partial order from the set of all
minors that has been introduced for the straightening law (see
Section \ref{SectStr}). The set of all minors is a distributive
lattice with respect to this order, and $\cA$ is a sublattice:
suppose that $r\le s$; to wit,
\begin{align*}
[a_1 \dots a_s]\wedge [b_1\dots b_r] &=[\min(a_1,b_1),
\min(a_2,b_2), \dots, \min(a_r,b_r), a_{r+1}, \dots, a_s],\\
[a_1 \dots a_s]\vee [b_1\dots b_r]&=[\max(a_1,b_1),
\max(a_2,b_2), \dots, \max(a_r,b_r)].
\end{align*}
For $a=[a_1\dots  a_s]\in \cA$ we set
$$
m_a=\ini([a_1\dots a_s])=\diag [a_1\dots a_s].
$$
For each $a\in\cA$ we introduce an indeterminate $p_{a}$ and
consider the $R$-algebra map
$$
\Phi: R[p_{a} :   a\in {\cA}  ]\to R[y_1,\dots, y_n],\qquad
\Phi(p_a)=m_ay_s.
$$

\begin{proposition}
\label{seemseasy}  $\Ker \Phi$ is generated by
\begin{itemize}
\item[(1)] the Hibi relations
    $$
    \underline{p_{a}p_{b}}-p_{a\wedge b}p_{a\vee b}
    $$
    with $a,b\in \cA$ incomparable, and
\item[(2)]  the relations of degree $1$ in the $p$'s,  more
    precisely, relations of the form
    $$
    \underline{x_{ij}p_{a}}-x_{ik}p_{b}
    $$
    with $a=[a_1\dots a_i, \dots, a_s]$, $a_{i-1}<j\le a_i$
    and $b=a\setminus \{a_i\} \cup \{j\}$.
\end{itemize}
These polynomials form a  Gr\"obner basis of $\Ker \Phi$ with
respect to every monomial order in which the underlined terms
are initial.
\end{proposition}

\begin{proof}
It is enough to prove that the given elements are a  Gr\"obner
basis of $\Ker \Phi$ .  The argument is quite standard (for
example, see for instance \cite[Chap.14]{Stu1} for similar
statements) and so we just sketch it. First note that a
monomial order selecting the underlined monomials is given by
taking  the reverse lexicographic order associated to a total
order  on the $p_{a}$'s that refines the partial order $\cA$.

To prove the assertion we choose an arbitrary  monomial in the
image of $\Phi$, say
$$
w y_{s_1}\cdots y_{s_e},\qquad s_1\geq \dots \geq s_e,\text{ $w$ a monomial in the
$x_{ij}$'s},
$$
 and check that the preimage $\Phi^{-1}(w
y_{s_1}\cdots y_{s_e})$ contains exactly one monomial of the
form
$$
up_{a_1}\cdots p_{a_e}
$$
with $|a_i|=s_i$ for $i=1,\dots, e$ and a monomial $u$ in the
$x_{ij}$'s such that
\begin{itemize}
\item[(i)] $a_1\leq a_2\leq \dots \leq a_e$  in the poset
    $\cA$;
\item[(ii)] for every $x_{ij}$ dividing $u$ and for every
    $k$, $1\leq k \leq e$, one has  either $j\geq a_{k,i}$
    or $j\leq a_{k,i-1}$ where $a_k=\{a_{k,1},\dots,
    a_{k,s_k}\}$ and, by convention, $a_{k,0}=0$.
 \end{itemize}
To check the claim one observes that $a_1$ is determined
uniquely as the minimum of the $b\in \cA$ such that $|b|=s_1$
and $m_b| w$, then $a_2$ is the minimum of the $b\in \cA$ such
that $|b|=s_2$ and $m_{a_1}m_b| w$ and so on.
\end{proof}

\begin{remark}
For every finite lattice $L$ one may consider the ring
$$K[L]=K[x : x\in L]/(  xy- (x\wedge y)(x\vee b) : x,y\in L).$$
Hibi proved in \cite{H} that $K[L]$ is a domain if and only if
$L$ is distributive and in that case $K[L]$ turns out to be
(isomorphic to) a normal semigroup ring. When $L$ is a
distributive lattice  $K[L]$ is called the Hibi ring of $L$.
That is why the elements  $p_{a}p_{b}-p_{a\wedge b}p_{a\vee b}$
in \ref{seemseasy} are called Hibi  relations.  In our setting
the Hibi ring associated to $\cA$ coincides with the
multi-graded coordinate ring of flag variety associated to the
sequence $1,2,\dots,m$ and also with the special fiber
$\Rees/(x_{ij})\Rees$  of the multi-Rees algebra
$\Rees(J_1,\dots,J_m)$.
\end{remark}

\begin{example}
For $m=n=4$ the generators of $\Ker \Phi$ are
$$\begin{array}{ccc}
        	x_{1,3}p_{4}- x_{1,4}p_{3} &
	x_{1,2}p_{4}- x_{1,4}p_{2} &
	x_{1,1}p_{4}- x_{1,4}p_{1} \\
	x_{1,2}p_{3}- x_{1,3}p_{2} &
	x_{1,1}p_{3}- x_{1,3}p_{1} &
	x_{1,1}p_{2}- x_{1,2}p_{1} \\
	x_{2,3}p_{24}- x_{2,4}p_{23} &
	x_{2,3}p_{14}- x_{2,4}p_{13}  &
	x_{2,2}p_{14}- x_{2,4}p_{12}  \\
	x_{1,2}p_{34}- x_{1,3}p_{24} &
	x_{1,1}p_{34}- x_{1,3}p_{14} &
	x_{1,1}p_{24}- x_{1,2}p_{14} \\
	x_{2,2}p_{13}- x_{2,3}p_{12} &
	x_{1,1}p_{23}- x_{1,2}p_{13} &
	x_{3,3}p_{124}- x_{3,4}p_{123} \\
	x_{2,2}p_{134}- x_{2,3}p_{124} &
	x_{1,1}p_{234}- x_{1,2}p_{134} &
	p_{34}p_{2}- p_{24}p_{3}  \\
	p_{34}p_{1}- p_{14}p_{3} &
	p_{24}p_{1}- p_{14}p_{2}  &
	p_{23}p_{1}- p_{13}p_{2}  \\
	p_{14}p_{23}- p_{13}p_{24} &
	p_{234}p_{1}- p_{134}p_{2} &
	p_{234}p_{14}- p_{134}p_{24} \\
	p_{234}p_{13}- p_{134}p_{23} &
	p_{234}p_{12}- p_{124}p_{23} &
	p_{134}p_{12}- p_{124}p_{13} \\
 \end{array}
$$
\end{example}

Now we have collected all arguments for our main result.

\begin{theorem}
\label{mainsagbi}\leavevmode
\begin{enumerate}
\item  The set $X\cup\cF$ is a Sagbi basis of the
    multi-Rees algebra $\Rees(J_1,\dots,J_m)$.
\item For all $a_1,\dots, a_m\in \NN$ we have
    $$
    \ini_\prec(J_1^{a_1}\cdots J_m^{a_m} )=
    \ini_\prec(J_1)^{a_1} \cdots \ini_\prec(J_m)^{a_m},
    $$
    and $J_1^{a_1}\cdots J_m^{a_m}$ has a linear
    resolution.
\item   $\Rees(J_1,\dots,J_m)$ is a normal and Koszul
    domain.
\end{enumerate}
\end{theorem}

\begin{proof}
(1) follows  from Proposition \ref{seemseasy}, Lemma
\ref{SagBuch} and Remark \ref{SagBuch1}, provided we can
``lift'' the Hibi relations. For incomparable $a,b\in \cA$
consider the non-standard product $[a][b]$. In its standard
representation we have only standard monomials with the same
shape. A standard monomial with super-standard row tableau can
be reconstructed from its initial (diagonal) term and the only
standard monomial with super-standard row with initial term
equal to that of  $[a][b]$ is $[a\wedge b][a\vee b]$. It
follows that $[a\wedge b][a\vee b]$ appears in the standard
representation of $[a][b]$ and all the other standard monomials
have leading term strictly smaller than that of $[a][b]$. This
shows that the Hibi relations lifts.

(2) The equation $\ini_\prec(J_1^{a_1}\cdots J_m^{a_m} )=
\ini_\prec(J_1)^{a_1} \cdots \ini_\prec(J_m)^{a_m}$ has already
been stated in Theorem \ref{mainkrs}, but it follows again from
the equivalence of \eqref{1} and \eqref{3}.

Note that Theorem \ref{mainkrs} conversely implies the
liftability of the Hibi relations since it shows that
$X\cup\cF$ is a Sagbi basis.

The algebra $\Rees(\ini_\prec(J_1),\dots,\ini_\prec(J_m))$ is
Koszul since it is defined by a Gr\"obner basis of quadrics as
stated in Proposition \ref{seemseasy}. But
$\ini_\prec(\Rees(J_1,\dots,J_m))=
\Rees(\ini_\prec(J_1),\dots,\ini_\prec(J_m))$, and the
Koszulness of $\Rees(J_1,\dots,J_m)$ is a consequence of the
preservation of Koszulness under Sagbi deformation
\cite[3.14]{BC2}. This proves part of (3) and Theorem \ref{Kos}
implies that the ideals $J_1^{a_1}\cdots J_m^{a_m}$ have a
linear resolution.

(3) Only the normality of the multi-Rees algebra is still open.
 To this end one can apply the preservation
of normality under Sagbi deformation \cite[3.12]{BC2} and apply
\cite[Prop.13.15]{Stu1} which implies that
$\ini_\prec(\Rees(J_1,\dots,J_m))$ is normal since its defining
ideal has a square-free initial ideal.
\end{proof}

\section{Equivariant $R$-modules}\label{SectK}
In this section we make the assumption that $m \geq n$. This
will simplify the conclusion of main result of the section,
which has a less pleasing analogue when $m < n$.

Let $T^m \subset \GL_m(K)$ denote the diagonal torus, and set
$G := T^m \times \GL_n(K)$. Then $G$ acts on $R$ as in
Section~\ref{Intro}. In this section we consider the
Grothendieck group of finitely generated $G$-equivariant
$R$-modules with a rational $G$-action, denoted $K^0_G(R)$.

Since $R$ is a polynomial ring, the group $K^0_G(R)$ can be
identified with the representation ring of $G$. Hence
$K^0_G(R)$ is generated by the free equivariant modules $R
\otimes V$, as $V$ ranges over all finite dimensional rational
$G$ modules. The group $K^0_G(R)$ inherits a product from the
tensor product of $G$-modules. The product of the classes of
two general equivariant $R$-modules can be expressed in terms
of their Tor-modules, a fact we will not need here.

Using the multigrading of Section~\ref{SectStr}, an equivariant
$R$-module $M$ is at once seen to be a multigraded module. We
write its Hilbert series as
\[
\hilb(M) = \sum_{\mathbf{a} \oplus \mathbf{b} \in \ZZ^m \oplus \ZZ^n}
\dim_K(M_{\mathbf{a} \oplus \mathbf{b}})u^\mathbf{a} v^{\mathbf{b}} \in
\ZZ[[u_1^{\pm 1},\dots,u_m^{\pm 1},v_1^{\pm 1},\dots,v_n^{\pm 1}]]^{\mathfrak{S}_n}.
\]
Here the group $\mathfrak{S}_n$ is permuting the $v$ variables,
and the $\GL_n(K)$-invariance of $M$ forces $\hilb(M)$ to be
invariant under this action. The Hilbert series $\hilb(M)$ can
alternately be described as the character of the $G$-module
$M$. There is a Laurent polynomial $K(M;u,v)$ such that
\[
\hilb(M) = \frac{K(M;u,v)}{\prod_{i=1}^m\prod_{j=1}^n (1-u_iv_j)}
\]
and hence we identify the class of a module $M$ in $K^0_G(R)$
with $K(M;u,v)$ \cite[Th.~8.20]{MS}. This makes the
identification of $K^0_G(R)$ with the representation ring of
$G$ explicit:
\[
K^0_G(R) = \ZZ[ u_1^{\pm 1},\dots,u_m^{\pm 1},v_1^{\pm 1},\dots,v_n^{\pm 1}]^{\mathfrak{S}_n},
\quad M \mapsto K(M;u,v).
\]

The superstandard bitableau of shape $S$ span a representation of $G$
\cite[Thm.~11.5(a)]{BV}. It follows that the ideals $J_S$ are
$G$-invariant, and hence the quotient ring $R/J_S$ defines an element
of $K^0_G(R)$. This stands in contrast to an ideal generated by
standard bitableaux with a fixed left tableau, which does not
necessarily a $G$-invariant ideal (see \cite[Rmk.~11.12]{BV}).

\begin{proposition}\label{propKbasis}
The classes of the modules $R/J_S$, as $S$ ranges over shapes
$S$ with part sizes at most $n-1$, freely generate $K^0_G(R)$
as a module over $\ZZ[u_1^{\pm 1},\dots,u_m^{\pm 1},(v_1 \cdots
v_n)^{\pm 1}]$.
\end{proposition}
Multiplication by $(v_1 \cdots v_n)^{ \pm 1}$ corresponds to
tensoring with the determinantal character of $\GL_n(K)$ or its
dual, and multiplication by a $u$ variable corresponds to
tensoring with a character of $T$.
\begin{proof}
  It is sufficient to show that the polynomials $K(J_S;u,v)$, as $S$
  ranges over all shapes, generate
  \[
  \ZZ[u_1^{\pm 1},\dots,u_m^{\pm 1},v_1^{\pm 1},\dots,v_n^{\pm
    1}]^{\mathfrak{S}_n}
  \]
  as a module over $\ZZ[u_1^{\pm 1},\dots,u_m^{\pm 1},(v_1 \cdots
  v_n)^{-1}]$. This is because all rational representations of $\GL_n(K)$
  are obtained by tensoring polynomial representations with a power of
  the determinantal representation, and $K(R/J_S;u,v) = 1-
  K(J_S;u,v)$.

  For any shape $S$ whose part sizes are at most $n$, let
  $\sigma_S(v)$ denote the \emph{Schur polynomial} in variables
  $v_1,\dots,v_n$. That is, $\sigma_S(v)$ will be the generating
  function in $v$ for the content of tableaux of shape $S$ with
  strictly increasing rows, weakly increasing columns and entries in
  $\{1,\dots,n\}$.

  The ideal $J_S$ is generated by an irreducible representation of $G$
  whose character is $u_1^{s'_1} \cdots u_m^{s'_\ell}\cdot
  \sigma_S(v)$, were $S'=s_1',\dots,s_\ell'$ denotes the transpose of
  $S$. This is the shape whose $j$th part is $s_j' = \#\{s_i : i \geq
  j\}$. It follows that
  \[
  \hilb(J_S) =  u_1^{s'_1} \cdots u_m^{s'_\ell}\cdot
  \sigma_S(v) + \cdots
  \]
  where the ellipsis denotes a $\ZZ[v]$-linear combination of Schur
  polynomials of degree larger than $\sum_i s_i$. Multiplying by
  $\prod_{i=1}^m \prod_{j=1}^n (1-u_i v_j)$, this proves that
  $K(J_S;u,v)$ takes the same form. We conclude the linear
  independence of the classes, since the Schur polynomials are
  linearly independent.

  To finish the proof, we must show that every Schur polynomial can be
  written as a finite $\ZZ[u^{\pm 1}]$-linear combination of these
  classes. The difficulty with this lays in demanding the finiteness
  of the expression. We will show, first, that the
  Schur polynomials appearing in $K(J_S;u,v)$ never get too long, and
  second, when $S=n,\dots,n$ ($\ell$-factors) that $K(J_S;u;v) =
  (u_1\cdots u_n)^\ell \sigma_S(v)$.

  We will use the fact that passing to an initial ideal does not alter
  $K$-classes: $K(J_S;u,v) = K(\ini(J_S);u,v)$
  \cite[Prop.~8.28]{MS}. Although $\ini(J_S)$ is no longer a
  $G$-equivariant ideal, we can compute its $K$-polynomial in the
  Grothendieck group of multigraded modules. To understand
  $K(\ini(J_S);u,v)$ we resolve the quotient $R/\ini(J_S)$ by its
  highly non-minimal Taylor resolution \cite[Ch.~6]{MS}. Write
  $\ini(J_S) = \langle m_1,\dots, m_r \rangle$, where the $m_i$ are
  the leading terms of the superstandard bitableaux of shape $S$ in
  the diagonal term order. Given a subset $I \subseteq\{1,\dots,r\}$,
  set $m_I = \operatorname{lcm} \{m_i : i \in I\}$. If the degree of
  $m_I$ is $(\mathbf{a}_I,\mathbf{b}_I) \in \NN^m \oplus \NN^n$, then
  the $i$th piece of the Taylor resolution of $\ini(J_S)$ is
  $\bigoplus_{I: |I|=i} R(-\mathbf{a}_I,-\mathbf{b}_I)$. It is a fact
  that this can be endowed with a differential yielding a resolution
  of $R/\ini(J_S)$.

  We claim that all Schur polynomials that appear with a non-zero
  coefficient in $K(J_S;u,v)$ have length at most $s_1'$. Suppose that
  this were not true. Writing $K(J_S;u,v)$ in the standard basis of
  monomials of $\ZZ[u,v]$ this implies that the variable $v_1$ appears
  with exponent greater than $s_1'$. However, appealing to the fact
  that the Taylor resolution can be used to compute $K(R/\ini(J_S);u,v)$, this
  means that there is some monomial $m_I$ whose associated degree
  $(\mathbf{a}_I,\mathbf{b}_I)$ has $(\mathbf{b}_I)_1 >
  s_1'$. However, the least common multiple of all the $m_i$ is of the
  form $x_{11}^{s_1'}\cdot$(a monomial in $x_{ij}$ with $j \neq 1$),
  which is a contradiction.

  It follows that $K(J_S;u,v)$ can be written as a finite
  $\ZZ[u]$-linear combination of Schur polynomials whose shape is
  contained in a $s_1' \times n$ box. Suppose that $S = n,\dots,n$. Then
  the ideal $J_S$ is principal, generated by a power of a maximal
  minor of $X$. That $K(J_S;u,v) = (u_1 \cdots u_n)^{s_1'} (v_1 \cdots
  v_n)^{s_1'}$ is immediate. By induction, we may write $\sigma_S(v)$
  as a linear $\ZZ[u^{\pm 1}]$-linear combination of the classes of
  ideals generated by superstandard tableaux.
\end{proof}

\begin{example}
  Take $n=m=3$ and $S=2,1$. The least common multiple of the initial
  monomials of the superstandard bitableaux of shape $S$ is
  $x_{11}^2x_{12}^2x_{22}x_{13}x_{23}$. Using Macaulay2, we have,
  \begin{multline*}
  K(J_S;u,v) = \sigma_{2,1}(v) u_1^2 u_2 -
  \sigma_{2,2}(v) u_1^3 u_2 -
  \sigma_{3,1}(v) (u_1^3 u_2 + u_1^2 u_2^2) \\+
  \sigma_{3,2}(v) (u_1^4 u_2 + u_1^3u_2^2) -
  \sigma_{3,3}(v) u_1^4 u_2^2.
  \end{multline*}
  Observe that each shape appearing has at most two parts.
\end{example}

\end{document}